\documentclass[a4paper,12pt,leqno]{article}
\usepackage{amsmath,amsfonts,amsthm,amssymb,dsfont}
\usepackage[alphabetic]{amsrefs}
\usepackage[OT4]{fontenc}
\usepackage{enumerate}
\usepackage{epsfig}
\usepackage{psfrag}

\long\def\symbolfootnote[#1]#2{\begingroup%
\def\thefootnote{\fnsymbol{footnote}}\footnote[#1]{#2}\endgroup}

\newtheorem{theorem}{Theorem}[section]
\newtheorem{lemma}[theorem]{Lemma}

\newtheorem{prop}[theorem]{Proposition}
\newtheorem{cor}[theorem]{Corollary}

\theoremstyle{definition}
\newtheorem{rem}[theorem]{Remark}
\newtheorem{defin}[theorem]{Definition}
\newtheorem{example}[theorem]{Examples}
\newtheorem{conj}[theorem]{Conjecture}
\renewcommand{\proof}{\medskip\par\noindent\textbf{Proof.} \ignorespaces}

\renewcommand{\qed}{\quad\hskip0pt\null\hfill$\square$\par}
\newcommand {\e}{\varepsilon}
\newcommand {\R} {{\mathbb{R}}}
\newcommand {\N} {{\mathbb{N}}}
\newcommand {\Z} {{\mathbb{Z}}}
\newcommand {\Tau} {\mathrm{T}}
\newcommand {\Q} {{\mathbb{Q}}}
\newcommand {\Sp} {S_{0,5}}
\newcommand {\PML} {\mathcal{PML}}
\newcommand {\AML} {\mathcal{AML}}
\newcommand {\EL} {\mathcal{EL}}
\renewcommand {\S} {\mathcal{S}}
\newcommand {\PEML}{\mathcal{MPML}}
\newcommand {\MPML}{\mathcal{MPML}}

\begin{document}

\begin{center}
\large\bfseries The ending lamination space of the five-punctured
sphere is the N\"obeling curve
\end{center}

\begin{center}\bf
Sebastian Hensel$^a$\symbolfootnote[1]{Supported by the Max Planck Institute for Mathematics in Bonn.} \& Piotr Przytycki$^b$\symbolfootnote[2]{Partially supported by MNiSW
grant N201 012 32/0718 and the Foundation for Polish Science.
}
\end{center}

\begin{center}\it
$^a$ Mathematisches Institut, Universit\"at Bonn,\\
 Endenicher Allee 60, 53115 Bonn, Germany\\
\emph{e-mail:}\texttt{loplop@math.uni-bonn.de}
\end{center}

\begin{center}\it
$^b$ Institute of Mathematics, Polish Academy of Sciences,\\
 \'Sniadeckich 8, 00-956 Warsaw, Poland\\
\emph{e-mail:}\texttt{pprzytyc@mimuw.edu.pl}
\end{center}

\begin{abstract}
\noindent We prove that the ending lamination space of the
five-punctured sphere is homeomorphic to the N\"obeling curve.
\end{abstract}



\section{Introduction}
Let $S_{g,p}$ be the orientable surface of genus $g$ with $p$
punctures. The set of essential simple closed curves on the surface
$S_{g,p}$ can be arranged into the curve complex
$\mathcal{C}=\mathcal{C}(S_{g,p})$. The combinatorics of this
intricate object is used in particular to study the mapping class
group of $S_{g,p}$. The curve complex is particularly useful in view
of the result of Masur and Minsky \cite[Theorem 1.1]{MM} which says
that $\mathcal{C}$ is hyperbolic in the sense of Gromov. Hence,
using the Gromov product (see \cite[Chapter III.H]{BH}), one can
define the Gromov boundary $\partial \mathcal{C}$ in the usual way.
Note that the boundary $\partial \mathcal{C}$ is in general not
compact, since the curve complex $\mathcal{C}$ is not locally
finite. Hence the topology of $\partial \mathcal{C}$ can be a rich
subject to explore.

Interestingly, the boundary $\partial \mathcal{C}$ arises naturally
also from another construction. Namely, Klarreich \cite[Theorem
1.3]{Kla} (see also \cite[Section 1]{Ham}) proved that $\partial
\mathcal{C}$ is homeomorphic, by an explicit homeomorphism, to the
ending lamination space $\EL=\EL(S_{g,p})$. To define this space, we
need to recall some standard notions.

We denote the space of geodesic laminations on $S_{g,p}$ as usual by
$\mathcal{L}=\mathcal{L}(S_{g,p})$. An \emph{ending lamination} is a
minimal filling geodesic lamination. The set of all ending
laminations is denoted by $\EL \subset \mathcal{L}$. It remains to
describe the topology on the set $\EL$. Let $\PML$ denote the space
of projective measured laminations. Let $\phi\colon \PML \rightarrow
\mathcal{L}$ be the measure forgetting map which maps a projective
measured lamination to its support. Note that $\phi$ is not
surjective, because there are geodesic laminations which do not
admit a transverse measure of full support. However, every ending
lamination admits a transverse measure of full support, so $\EL$ is
contained in the image of $\phi$. Let $\PEML\subset \PML$ be the
preimage of $\EL$ under $\phi$. The space $\EL$ is naturally
equipped with the topology induced from $\PEML$ by the quotient map
$\phi$.

Here is a short account on what is known about the ending lamination
space $\EL$ depending on $g$ and $p$. We call $\xi = 3g-3+p$ the
\emph{complexity} of the surface. The cases where the complexity is
at most one are easy and well-known, in particular for $\xi=1$
(i.e.\ in the case where the surface is the four-punctured sphere or
the once-punctured torus) we have $\EL\simeq\R\setminus \Q$. Assume
now $\xi>1$, i.e.\ the surface in question is hyperbolic and
non-exceptional. In this case Gabai \cite{G} showed that $\EL$ is
connected, locally path-connected, and cyclic, which concluded a
series of results in this direction. Previously Leininger--Schleimer
\cite{LS} proved that $\EL$ is connected, provided $g\geq 4$, or
$g\geq 2$ and $p\geq 1$. Moreover, Leininger--Mj--Schleimer
\cite{LMS} proved that if $g\geq 2$ and $p=1$, then $\EL$ is locally
path connected.

In spite of the fact that so little about $\EL$ is known so far, we
were encouraged by Mladen Bestvina to address the following.

\begin{conj}
\label{conjecture} The ending lamination space of $S_{g,p}$ is
homeomorphic to the $(\xi-1)$--dimensional N\"obeling space.
\end{conj}

\begin{defin}
\label{definition noebeling} The \emph{$m$--dimensional N\"obeling
space} $N^{2m+1}_m$ is the topological space obtained from
$\R^{2m+1}$ by removing all points with at least $m+1$ rational
coordinates.
\end{defin}

In this terminology, the ending lamination space of the
four-punctured sphere (and the once-punctured torus) is homeomorphic
to the $0$--dimensional N\"obeling space. This agrees with
Conjecture~\ref{conjecture}.

The \emph{N\"obeling curve} is the $1$--dimensional N\"obeling
space, i.e.\ the topological space obtained from $\R^3$ by removing
all points with at least two rational coordinates. The main result
of this article is to confirm Conjecture~\ref{conjecture} in the
following case.

\begin{theorem}
\label{main} The ending lamination space of the five-punctured
sphere is homeomorphic to the N\"obeling curve.
\end{theorem}

Since $\EL$ is homeomorphic to the Gromov boundary of the curve
complex, which is the same (as a simplicial complex) for the
twice-punctured torus and the five-punctured sphere (see \cite[Lemma
2.1(a)]{Luo}), we have the following.

\begin{cor}
The ending lamination space of the twice-punctured torus is
homeomorphic to the N\"obeling curve.
\end{cor}

\medskip
The article is organised as follows. In Section~\ref{section The
Nobeling curve} we provide a topological characterisation of the
N\"obeling curve which we use in the proof of Theorem~\ref{main}. In
Section~\ref{section The ending lamination space}, using train
tracks, we choose a convenient neighbourhood basis for the ending
lamination space. Then, in Section~\ref{section Almost filling
paths}, we give an account on Gabai's method for constructing paths
in $\EL$.

The proof of Theorem~\ref{main} splits into two parts. In Section
\ref{section Outline of the proof} we begin the proof and in
particular we obtain a dimension bound. The main part of the proof
is to obtain a universality property, with which we conclude in
Section~\ref{section Universality}.

\medskip

We thank Andrzej Nag\'orko for discussions on the N\"obeling curve
and for the characterisation in Section~\ref{section The Nobeling
curve}; Ursula Hamenst\"adt for discussions on the ending lamination
space and for suggestions on how to improve our preprint; and
Mladen Bestvina and Saul Schleimer for encouragement. This work was
partially carried out during the stay of the second author at the
Hausdorff Institute of Mathematics in Bonn.

\section{The N\"obeling curve}
\label{section The Nobeling curve}

In this section we give a useful characterisation of the N\"obeling
curve following from Kawamura--Levin--Tymchatyn~\cite{KLT}. We
learned about this characterisation and the way to derive it from
\cite{KLT} using a standard topological argument from Andrzej
Nag\'orko.

\begin{theorem}
\label{characterization_easy} If a Polish space of dimension $1$ is
connected, locally path connected and satisfies the locally finite
$1$--discs property, then it is homeomorphic to the N\"obeling
curve.
\end{theorem}

Recall that a topological space is \emph{Polish} if it is separable
and admits a complete metric. A topological space has
\emph{dimension $0$} (resp.\ \emph{dimension at most $m$}, for
$m>0$) if each point has a basis of neighbourhoods with empty
boundaries (resp.\ with boundaries of dimension at most $m-1$). A
space has \emph{dimension $m$}, for $m>0$, if it has dimension at
most $m$, but does not have dimension $m-1$. In case of a Polish
space this coincides with the usual covering dimension (see
\cite[Theorem 1.7.7]{Eng}).

\begin{defin}
A topological space $X$ satisfies the \emph{locally finite
$m$--discs property} if we have the following. For any family of
continuous maps $f_n \colon I^m=[0,1]^m \rightarrow X$, where $n\in
\N$, and any open cover $\mathcal{U}$ of $X$, there are continuous
maps $g_n\colon I^m \rightarrow  X$ such that
\begin{enumerate}[(i)]
\item
for each $x\in X$ there is a neighbourhood $U\ni x$ satisfying
$g_n(I^m)\cap U=\emptyset$ for sufficiently large $n$,
\item
for each $t\in I^m, n\in \N,$ there is $U\in \mathcal{U}$ such that
both $f_n(t)$ and $g_n(t)$ lie in $U$ (we say that such $f_n$ and
$g_n$ are \emph{$\mathcal{U}$--close}).
\end{enumerate}
If additionally $g_n(I^m)$ may be required to be pairwise disjoint,
then $X$ satisfies the \emph{discrete $m$--discs property}.
\end{defin}

In the remaining part of this section we explain how to derive
Theorem~\ref{characterization_easy} from the following.

\begin{theorem}[{\cite[Theorem 2.2]{KLT}}]
\label{characterization_hard} A $1$--dimensional Polish space is the
N\"obeling curve if and only if it is an absolute extensor in
dimension $1$ and strongly universal in dimension $1$.
\end{theorem}

In fact, in order to address Conjecture~\ref{conjecture} in the future, 
we have decided to discuss the higher
dimensional analogue of Theorem~\ref{characterization_hard}.

\begin{theorem}[{\cite[Topological rigidity theorem]{N}}]
\label{characterization_hard multidim} An $m$--dimensional Polish
space is the $m$--dimensional N\"obeling space if and only if it is
an absolute extensor in dimension $m$ and strongly universal in
dimension $m$.
\end{theorem}

A metric space $X$ is an \emph{absolute extensor in dimension $m$},
if every continuous map into $X$ from a closed subset of an at most
$m$--dimensional metric space extends over the entire space. Assume
now that $X$ is \emph{locally $k$--connected} for every $k<m$ (see
\cite[Definition 3.1]{Du}, for $m=1$ this means that $X$ is locally
path connected). In that case, by Dugundji \cite[Theorem 9.1]{Du},
$X$ is an absolute extensor in dimension $m$ if and only if all of
its homotopy groups in dimension less than $m$ vanish. For $m=1$
this means that $X$ is connected. Summarizing, if a metric space is
locally $k$--connected for every $k<m$, and all of its homotopy
groups in dimension less than $m$ vanish, then it is an absolute
extensor in dimension $m$. In particular, if a metric space is
connected and locally path connected, then it is an absolute
extensor in dimension $1$.

A Polish space $X$ is \emph{strongly universal} in dimension $m$ if
any continuous map $f\colon Y \rightarrow X$ from an at most
$m$--dimensional Polish space $Y$ to $X$ is \emph{approximable} by
closed embeddings. This means that for any open cover $\mathcal{U}$
of $X$ there is a closed embedding $g\colon Y \rightarrow  X$ such
that $f$ and $g$ are $\mathcal{U}$--close. We discuss below, under
what hypothesis strong universality in dimension $m$ follows from
the locally finite $m$--discs property.

By \cite[discussion after Theorem 2.4]{Cu}, any Polish space $X$
satisfying the locally finite $m$--discs property satisfies also the
discrete $m$--discs property. Bowers \cite[Theorem in Appendix, part
(2)]{Bo} proves that the latter implies strong universality in
dimension $m$, under the hypothesis that $X$ is an ANR. Recall that
a topological space $X$ is an \emph{absolute neighbourhood retract}
(shortly, an \emph{ANR}) if for each closed subset $A\subset X$,
which is normal, there is an open neighbourhood $U\subset X$ such
that $A$ is a retract of $U$. Unfortunately, N\"obeling spaces are
not ANR, hence Bowers' theorem as stated is not sufficient for our
purposes. However, his proof yields the following.

\begin{theorem}
\label{Bowers} Let $X$ be a Polish space which is locally
$k$--connected for all $k<m$. If $X$ satisfies the discrete $m$--discs
property, then it is strongly universal in dimension $m$.
\end{theorem}

In other words, we can replace the ANR hypothesis in 
\cite[Theorem in Appendix]{Bo} by local $k$--connectedness for all $k<m$.
Indeed, the only two places in the proof, where the ANR hypothesis
is used, are lines 1 and 5 on page 129, in the proof of Lemma C.
However, in both cases the argument only requires the following
property (which is satisfied if $X$ is an ANR). Namely, let $k<m$ and
let $S^{k}$ be the $k$--sphere. Bowers' argument requires that for
every open cover $\mathcal{U}$ of $X$, there is a refinement
$\mathcal{U}'$, such that if $f_0,f_1\colon S^{k} \rightarrow X$ are
$\mathcal{U}'$--close, then there is a homotopy between $f_0$ and
$f_1$ with each track contained in some $U\in \mathcal{U}$. By
\cite[Theorem 5.1]{Du} this property follows from local
$k$--connectedness. This concludes the argument for Theorem~\ref{Bowers}.
\medskip

By Theorems~\ref{characterization_hard multidim} and~\ref{Bowers}
and by the preceding discussion we conclude with the following,
which in the case of $m=1$ amounts exactly to
Theorem~\ref{characterization_easy}.

\begin{cor}
\label{characterization_easy multidim} Let $X$ be Polish space of
dimension $m$ which is locally $k$--connected for every $k<m$, and
all of whose homotopy groups in dimension less than $m$ vanish.
Assume that $X$ satisfies the locally finite $m$--discs property.
Then $X$ is homeomorphic to the $m$--dimensional N\"obeling space.
\end{cor}

\section{Train track partitions}
\label{section The ending lamination space} Our strategy of proving
Theorem~\ref{main} is to use the topology of $\PML$, the space of
projective measured laminations, to obtain information about the
topology of the ending lamination space $\EL$. To this end, we
construct a sequence of finer and finer partitions of $\PML$ into
polyhedra using Thurston's notion of train tracks (see \cite[Section
8.9]{Th}). We then show that these polyhedra project to a convenient
neighbourhood basis of $\EL$.

For a thorough treatment of train tracks, as well as the basic
definitions, we refer the reader to the book of Penner and Harer
\cite{PH}. Note however that, in contrast to the treatment in
\cite{PH}, for us every train track is \emph{generic} (i.e.\ each
switch is at most trivalent). In the following, we briefly recall
some definitions and statements important to the current work.

Let $\tau$ be a recurrent train track. We denote by $P(\tau)$ the
\emph{polyhedron of projective measures} of $\tau$, that is the set
of all projective measured laminations which are carried by $\tau$.
$P(\tau)$ has the structure of an affine polyhedron, where the faces of
$P(\tau)$ correspond to recurrent proper subtracks $\tau' < \tau$
(see \cite[pp.\ 116--117]{PH}). The inclusion map $P(\tau) \subset
\PML$, where $P(\tau)$ is equipped with the topology coming from the
polyhedral structure, is continuous. In particular, for any train
track $\tau$ the polyhedron of projective measures $P(\tau)$ is a
closed set in $\PML$. The \emph{interior of $P(\tau)$} is its interior 
with respect to the polyhedral structure, i.e.\ the set of transverse 
measures which put positive mass on each branch of $\tau$. Note that 
in general this is not the interior of the set $P(\tau) \subset \PML$ with
respect to the topology of $\PML$. In the sequel we denote the
interior of the polyhedron of projective measures by $V(\tau)\subset
\PML$. We denote the boundary $P(\tau) \setminus V(\tau)$ of the
polyhedron of projective measures by $\partial V(\tau)$. From now
on, the expression \emph{boundary of $X$} will always mean
$\mathrm{Fr} X=\overline{X}\setminus \mathrm{int} X$ (the boundary
in the topological sense). Note that in this terminology $\partial
V(\tau)$ might not be the boundary of $V(\tau) \subset \PML$. Let
$$U(\tau) = \phi(V(\tau) \cap \MPML)$$ (equivalently, $U(\tau)$ is the
set of ending laminations which are fully carried by $\tau$). We
denote the inverse correspondence between (families of) these sets
by $\Psi$, i.e.\ $\Psi(U(\tau)) = V(\tau)$.

Unless stated otherwise, from now on we restrict to the case of
$\Sp$, where $\PML$ is $3$--dimensional. We call a train track
$\eta$ \emph{complete} if it is recurrent, transversely recurrent
and maximal. Recall that if $\eta$ is complete, then $V(\eta)$ is
$3$--dimensional, hence open in $\PML$ (see e.g.\ \cite[Lemma
3.1.2]{PH}) and consequently $U(\eta)$ is open in $\EL$. In
particular we have $\partial V(\eta) = \mathrm{Fr}V(\eta)$. We call
a train track $\sigma$ \emph{nearly complete}, if it is birecurrent,
carries an ending lamination and $P(\sigma)$ is $2$--dimensional (in
particular $\sigma$ is not complete).

 \begin{rem}
   \label{no crossing boundaries}
   Let $\mu_0, \mu_1$ be measured geodesic
   laminations which do not intersect transversally. Suppose that for some train track $\tau$ its polyhedron of projective measures $P(\tau)$ contains a projective
   class of $\mu_t = (1-t)\mu_0 + t\mu_1$ for
   some $t \neq 0, 1$. Then the whole interval $\{\mu_t\}_{t \in [0,1]}$ projects into $P(\tau)$. This is because the support of $\mu_t$
equals $\phi(\mu_1) \cup \phi(\mu_2)$ except maybe for $t=0$ or $1$, and projective measured laminations are carried by train tracks if
and only if their supports are.
 \end{rem}

We will need the following lemma, which shows how $\MPML$ can
intersect the polyhedron of projective measures of a complete or
nearly complete train track.
\begin{lemma}
\label{key lemma}~\\
\vspace{-.5cm}
  \begin{enumerate}[(i)]
  \item Let $\sigma$ be a nearly complete train track. Then $\partial V(\sigma)$ contains no filling
    lamination. In particular, $\partial V(\sigma)$ is disjoint from $\MPML$.
  \item Let $\eta$ be a complete train track. Then the $1$--skeleton of $\partial V(\eta)$ contains
    no filling lamination. In particular, the intersection of the $1$--skeleton of $\partial V(\eta)$ with $\MPML$ is empty.
  \item Let $\sigma$ be a nearly complete train track. Then $U(\sigma)$ is closed in $\EL$.
  \end{enumerate}
\end{lemma}

\begin{proof}
  \begin{enumerate}[(i)]
  \item Let $\sigma$ be a nearly complete train track. Recall that $\partial V(\sigma)$ is the union
    of $P(\tau)$ over all recurrent proper subtracks $\tau < \sigma$. We show that no proper subtrack
    of $\sigma$ is filling, which immediately implies assertion (i).

    Since $\sigma$ is nearly complete, it carries a filling lamination. Hence its complementary regions
    are topological discs or once-punctured discs. Thus, on a five-punctured sphere a nearly complete
    train track has at least five complementary regions. Each of those regions gives a
    contribution of at least $-\frac{1}{2}$ to the (generalised) Euler characteristic of $\Sp$, with
    a contribution of exactly $-\frac{1}{2}$ if and only if the component is a triangle
    or a once-punctured monogon.

    If $\sigma$ had more than five complementary regions, the fact that $\chi(\Sp)$ equals $-3$ would imply
    that these regions have to be five once-punctured monogons and one triangle --- which
    would mean that $\sigma$ was complete.

    Hence, $\sigma$ has four once-punctured monogons and one once-punctured bigon as complementary
    regions. A proper subtrack $\tau < \sigma$ needs to erase at least one branch of $\sigma$, and hence
    join two of these regions (or one to itself). Thus some complementary region of $\tau$
    contains either two punctures or an essential curve --- hence $\tau$ is not filling.
  \item Let $\eta$ be complete. The $1$--skeleton of $\partial V(\eta)$ is the
    union of $P(\tau)$ over recurrent $\tau<\eta$ which are obtained from $\eta$ by removing at least
    two branches. Now assertion (ii) follows with the same Euler characteristic argument as in the
    proof of assertion (i).
  \item Let $\sigma$ be a nearly complete train track. The polyhedron of projective measures $P(\sigma)$ is
    a closed set in $\PML$. Since the topology of $\EL$ is induced by the map $\phi$, the set
    $\phi(P(\sigma) \cap \MPML)$ is closed in $\EL$. However, by assertion (i), we have
    $P(\sigma) \cap \MPML = V(\sigma) \cap \MPML$ and hence $U(\sigma)=\phi(V(\sigma)\cap\MPML)$ is closed in $\EL$.
  \end{enumerate}
  \qed
\end{proof}

\medskip
In the remaining part of this section, our aim is to find a
convenient neighbourhood basis for $\EL$, determined by a certain
set of train tracks.

\begin{defin}
\label{track basis partition}
Let $\Tau$ be a finite collection of complete
train tracks and let $\Sigma$ be a finite collection of nearly complete train tracks.
The pair $(\Tau, \Sigma)$ is called a \emph{train track partition} if
all $V(\tau)$ are pairwise
disjoint, for $\tau \in \Tau \cup \Sigma$, and together cover all of $\MPML$.

Note that in particular all $U(\tau)$ are pairwise
disjoint, for $\tau \in \Tau \cup \Sigma$, and cover all of $\EL$.
\end{defin}

\begin{example}
\label{examples}~\\
\vspace{-.5cm}
  \begin{enumerate}[(i)]
  \item
  Let $P$ be any pants decomposition for $\Sp$. Let $\Tau$ be the set of complete
  \emph{standard train tracks} with respect to $P$ (see \cite[Section 2.6]{PH}) and let $\Sigma$
    be the set of their nearly complete subtracks.

    We claim that $(\Tau, \Sigma)$ is a train track partition. To this end, first note that the $P(\eta)$,
    for $\eta \in \Tau$,
    cover all of $\PML$ and each projective measured lamination $\lambda$ is fully carried by a unique
    subtrack $\tau$ of one of the $\eta \in \Tau$ (see \cite[Sections 2.7 and 2.8]{PH}).
    In particular, $V(\tau)$ are disjoint for all $\tau \in \Tau \cup \Sigma$.
    By Lemma~\ref{key lemma}(ii), every $\lambda \in \MPML$ lies in $V(\tau)$ for
    some $\tau \in \Tau \cup \Sigma$.

    We call such a pair $(\Tau,\Sigma)$ a \emph{standard train track partition}.

  \item
  Let $(\Tau, \Sigma)$ be a train track partition, and let $\eta \in T$ be a complete train track. Denote
  by $\eta_L$ (and $\eta_R$) the left (respectively right) split of $\eta$ along some large branch $b$.
  Note that splitting $\eta$ amounts to cutting $P(\eta)$ along a hyperplane, so that we have
  $P(\eta) = P(\eta_L) \cup P(\eta_R)$ and $P(\eta_L) \cap P(\eta_R) = P(\sigma)$ for
  a common subtrack $\sigma$ of $\eta_L$ and $\eta_R$.

  If both $\eta_L$ and $\eta_R$ are complete, define $\Tau'$ by replacing $\eta \in T$ by
  $\{\eta_L, \eta_R\}$. Then, if $\sigma$ is nearly complete, add $\sigma$ to $\Sigma$ to obtain $\Sigma'$. If only one of the two train tracks $\eta_L$ and $\eta_R$ is
  complete, replace $\eta \in \Tau$ by this
  train track to get $\Tau'$ and set $\Sigma' = \Sigma$.

  Note that in both cases the resulting pair $(\Tau', \Sigma')$ is a train track partition.
  We say that $(\Tau', \Sigma')$ is obtained from $(\Tau, \Sigma)$ by
  a \emph{complete splitting move along $b$}.

  \item
  Let $(\Tau, \Sigma)$ be a train track partition. Let $\sigma \in \Sigma$ be any nearly
  complete train
    track. Consider the left (respectively right) split $\sigma_L$ (and $\sigma_R$) along some large
    branch $b$. As above we have
    $P(\sigma) = P(\sigma_L) \cup P(\sigma_R)$ and $P(\sigma_L) \cap P(\sigma_R) = P(\tau)$ for
    a common subtrack $\tau$ of $\sigma_L$ and $\sigma_R$. If both $\sigma_L$ and $\sigma_R$ are nearly complete,
    define $\Sigma'$ by replacing $\sigma \in \Sigma$ by $\{\sigma_L, \sigma_R\}$. Otherwise, replace $\sigma$
    by the train track $\sigma_L$ or $\sigma_R$ which is nearly complete.
    Note that in both cases, by Lemma~\ref{key lemma}(i), the resulting pair $(\Tau, \Sigma')$ is a train
    track partition.
    We say that $(\Tau, \Sigma')$ is obtained from $(\Tau, \Sigma)$ by a \emph{nearly complete
    splitting move along $b$}.
  \end{enumerate}
\end{example}

We now use the above examples to obtain the following.

\begin{theorem}
  \label{tracks form basis}
  There exists a sequence $\S = ((\Tau_k, \Sigma_k))_{k=0}^\infty$ of train track partitions satisfying the
  following two properties.
  \begin{description}
  \item[(Subdivision)] Let $K\geq k \geq 0$. For each $\eta \in \Tau_{K}$ there is an $\eta' \in \Tau_k$ satisfying $V(\eta) \subset
  V(\eta')$. For each $\sigma \in \Sigma_{K}$ there is a $\tau \in \Tau_k \cup
\Sigma_k$ satisfying
    $V(\sigma) \subset V(\tau)$.
  \item[(Fineness)] For each ending lamination $\lambda$ and each open set $W$ in $\PML$ containing
    $\phi^{-1}(\lambda)$ there is an open set $V$ in $\PML$ satisfying $W \supset V \supset \phi^{-1}(\lambda)$ of the
    following form.

    Either $V = V(\eta)$ with $\eta \in \Tau_k$, for some $k \geq 0$, or
    $V = V(\eta_1) \cup V(\sigma) \cup V(\eta_2)$ with $\eta_i \in \Tau_{k_i}, \sigma \in \Sigma_k$, for
    some $k_1,k_2,k\geq 0$. In the latter case we additionally require
    $P(\sigma) \subset \partial V(\eta_1) \cap \partial V(\eta_2)$ (see Figure~\ref{fig:triple_neighbourood}).

      We denote by ${\cal V}(\S)$ the family of all open sets $V$ of above type.
  \end{description}
\end{theorem}
\begin{figure}[htbp!]
  \centering
   \psfrag{v1}{$V(\eta_1)$}
   \psfrag{v2}{$V(\sigma)$}
   \psfrag{v3}{$V(\eta_2)$}
   \psfrag{v}{$V$}
  \includegraphics[width=0.5\textwidth]{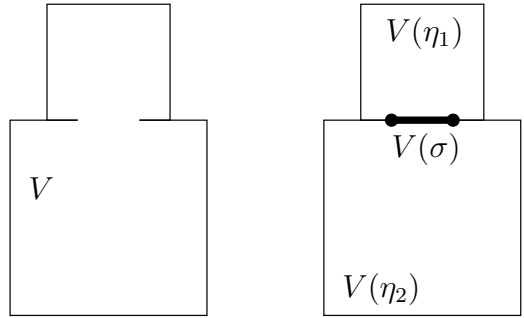}
  \caption{The case where $V = V(\eta_1) \cup V(\sigma) \cup V(\eta_2)$; we draw $V(\sigma)$ thick just to emphasise its presence; note that for simplicity the configuration is depicted in dimension $2$ instead of $3$.}
  \label{fig:triple_neighbourood}
\end{figure}
Before we begin the proof of Theorem~\ref{tracks form basis}, we
record the following.

\begin{rem}
\label{subdiv} If we have a sequence of train track partitions so
that each $(\Tau_{k+1},\Sigma_{k+1})$ is obtained from
$(\Tau_{k},\Sigma_{k})$ by a complete splitting move or a nearly
complete splitting move (see Examples~\ref{examples}(ii,iii)), then
they satisfy property (Subdivision). Moreover, property
(Subdivision) is preserved under passing to a subsequence.
\end{rem}

\begin{defin}
\label{partition sequence} We call sequences $\S = ((\Tau_k,
\Sigma_k))_{k=0}^\infty$ satisfying (Subdivision) and (Fineness)
\emph{good partition sequences}. In this terminology Theorem~\ref{tracks form basis}
says that there exists a good partition sequence.
\end{defin}

\begin{rem}
\label{discarding} If $((\Tau_k, \Sigma_k))_{k=0}^\infty$ is a good
partition sequence, then for any $K\geq 0$ the sequence $((\Tau_k,
\Sigma_k))_{k=K}^\infty$ is also a good partition sequence.
\end{rem}

Properties (Subdivision) and (Fineness) have the following immediate
consequences for the sets $U(\tau)$ in the ending lamination space.

\begin{cor}
  \label{properties of Us}
  Let $\S = ((\Tau_k, \Sigma_k))_{k=1}^\infty$ be a good partition sequence. Then (Subdivision) holds
  after replacing each $V(\tau)$ with $U(\tau)$.
Furthermore, let ${\cal U}(\S) = \{\phi(V \cap \MPML)\}_{V \in {\cal
V}(\S)}$.
  Then ${\cal U}(\S)$ is a neighbourhood basis of $\EL$.
\end{cor}

We denote by $\Psi\colon {\cal U}(\S) \rightarrow {\cal V}(\S)$ the
map extending the map $\Psi(U(\eta))=V(\eta)$ by $\Psi(U(\eta_1)\cup
U(\sigma) \cup U(\eta_2))=\Psi(U(\eta_1))\cup \Psi(U(\sigma)) \cup
\Psi(U(\eta_2))$.

\medskip
The remaining part of this section is devoted to the proof of
Theorem~\ref{tracks form basis}. We need to recall some facts about
full splitting sequences.

Let $b_1, \ldots, b_l$ be the large branches of a train track
$\tau$. Note that, if $\tau'$ is obtained from $\tau$ by a split at
$b_i$, every $b_j$ is still a large branch of $\tau'$ for $j\neq i$.
A \emph{full split of $\tau$} (see \cite[Section 5]{Ham_MCG1}) is a
train track which is obtained from $\tau$ by splitting at each large
branch $b_i$ exactly once (we also say that this train track is \emph{obtained
from $\tau$ by a full split}). A \emph{full splitting sequence} is a
sequence of train tracks $(\tau^i)_i$ such that $\tau^{n+1}$ is
obtained from $\tau^n$ by a full split. For an ending lamination
$\lambda$ carried by $\tau$, a \emph{full $\lambda$--splitting
sequence of $\tau$} is a full splitting sequence $(\tau^{i})_i$ with
$\tau^0 = \tau$ and such that each $\tau^{i}$ carries $\lambda$.

The following immediate consequence of \cite[Theorem 8.5.1]{Mos} is
the central part of the upcoming proof of Theorem~\ref{tracks form
basis}. (A similar theorem is obtained in \cite{Ham_MCG1}.)
\begin{theorem}
  \label{nesting}
  Let $\lambda$ be an ending lamination and let $(\tau^i)_i$ be a full $\lambda$--splitting sequence
  of some train track $\tau$. Then we have
  $$\bigcap_{i=1}^\infty P(\tau^i) = \phi^{-1}(\lambda).$$
In particular, for any open neighbourhood $W$ of
$\phi^{-1}(\lambda)$ in $\PML$, there is some
    $i_0 > 0$ such that for all $i > i_0$ we have $P(\tau^i) \subset W$.
\end{theorem}

\medskip\par\noindent\textbf{Proof of Theorem~\ref{tracks form basis}.} \ignorespaces
  Let $P$ be a pants decomposition for $\Sp$ and let $(\Tau_0, \Sigma_0)$ be the associated standard
  train track partition. We now describe an inductive procedure for building $(\Tau_k,
  \Sigma_k)$, where $k\geq 0$,
  which will satisfy property (Subdivision) and the following two additional properties.
  \begin{description}
  \item[(Adjacency)] If $\mu\in \MPML$ lies in $V(\sigma)$, where $\sigma\in \Sigma_k$ for some $k\geq 0$, then there are
  $\eta\neq\eta'\in \Tau_k$ such that $\mu$ lies in  $P(\eta)\cap P(\eta')$.
  Moreover, the set obtained from $$V(\eta)\cup
  V(\eta')\cup (\partial V(\eta)\cap \partial
  V(\eta'))$$ by removing the $1$--skeleta of $P(\eta)$ and
  $P(\eta')$ is an open neighbourhood of $\mu$ (see Figure~\ref{fig:adjecency}).
\begin{figure}
  \centering
  \includegraphics[width=0.4\textwidth]{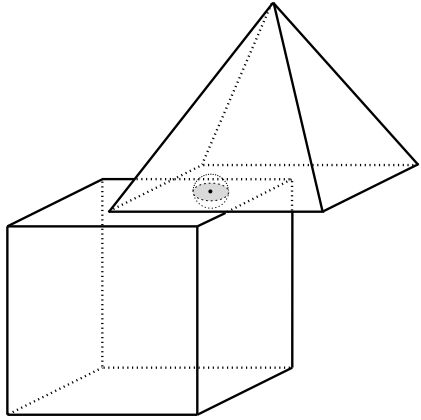}
  \caption{The open neighbourhood as in property (Adjacency).}
  \label{fig:adjecency}
\end{figure}

  \item[(Full splits)] If $\tau \in \Tau_k \cup \Sigma_k$ for some $k \geq 0$ and $\tau'$ is
    a complete or nearly complete train track obtained from $\tau$ by a full split, then
    $\tau'$ belongs to $\Tau_{k+1} \cup \Sigma_{k+1}$.
  \end{description}

Note that the standard train track partition $(\Tau_0, \Sigma_0)$
satisfies property (Adjacency). Moreover, by Lemma~\ref{key
lemma}(ii), a train track partition obtained by a complete splitting
move or a nearly complete splitting move (see Examples
\ref{examples}(ii,iii)) from another train partition satisfying
property (Adjacency), satisfies property (Adjacency) itself.

\medskip

We now describe our inductive procedure. Suppose that the train
track partition $(\Tau_k, \Sigma_k)$ has already been defined.
Roughly speaking we now perform the following operation. For each
$\eta \in \Tau_k$ we use complete splitting moves along all large
branches of $\eta$ to obtain $(\Tau_k', \Sigma_k')$. In the second
step we perform nearly complete splitting moves for each $\sigma \in
\Sigma_k$ along all large branches of $\sigma$ to obtain
$(\Tau_{k+1}, \Sigma_{k+1})$.

More precisely, let $\eta \in \Tau_k$ be a complete train track.
Denote the large branches of $\eta$ by $b_1, \ldots, b_l$. We now
perform a complete splitting move along $b_1$ to obtain a new
partition $(\Tau^1_k, \Sigma_k^1)$. The set $\Tau^1_k$ contains one
or two train tracks corresponding to the two possible splits of
$\eta$ along $b_1$. Each of those still contains $b_2, \ldots, b_l$
as large branches. We perform complete splitting moves along $b_2$
for those (one or two) train tracks in $\Tau^1_k$ to obtain the partition $(\Tau_k^2,
\Sigma_k^2)$. The set $\Tau_k^2$ contains now up to four train
tracks corresponding to the possible splits of $\eta$ along $b_1$
and $b_2$. We split all these (up to four) train tracks along $b_3$ and continue
this way for all large branches $b_1, \ldots, b_l$ until we
terminate with $(\Tau_k^l, \Sigma_k^l)$.

Note that this partition now contains every full split $\eta^1$ of
$\eta$, if $\eta^1$ is a complete train track. Moreover, we have $\Tau_k \setminus \{\eta\}\subset
\Tau^l_k$. We now repeat the same procedure for all $\eta' \in
\Tau_k\setminus\{\eta\}$. The resulting partition $(\Tau_k',
\Sigma_k')$ contains for every $\eta \in
\Tau_k$ all of its full splits which are complete.

In the second step, we obtain $(\Tau_{k+1}, \Sigma_{k+1})$ from
$(\Tau'_k, \Sigma'_k)$ by performing the analogous operation with
all elements of $\Sigma_k\subset \Sigma'_k$. Note that
$\Tau_{k+1}=\Tau'_k$. This completes the definition of
$\mathcal{S}=((\Tau_k, \Sigma_k))_{k=0}^\infty$.

\medskip

Since all $(\Tau_k, \Sigma_k)$ are obtained from $(\Tau_0,
\Sigma_0)$ by a sequence of complete splitting moves and nearly complete splitting
moves, $\S$ satisfies property (Adjacency). By Remark~\ref{subdiv},
$\S$ satisfies also property (Subdivision). Furthermore, by
construction, $\S$ satisfies property (Full splits). We now use this
information to derive property (Fineness).

\medskip

Let $\lambda \in \EL$, and let $W \subset \PML$ be any open set
containing $\phi^{-1}(\lambda)$. Let $(\eta^i)_i$ be a full
$\lambda$--splitting sequence of some $\eta^0\in \Tau_0$ carrying
$\lambda$. By property (Full splits), we have $\eta^k\in \Tau_k$. By
Theorem~\ref{nesting}, for sufficiently large $k$ we have
$V(\eta^k)\subset W$. Hence if $\phi^{-1}(\lambda)\subset
V(\eta^k)$, we are done. Otherwise, $\phi^{-1}(\lambda)$ is
contained in $\partial V(\eta^k)$.

Then the strategy is roughly speaking the following. We consider the
polyhedron $P(\eta'^k)$ ``on the other side'' of the face of
$P(\eta^k)$ containing $\phi^{-1}(\lambda)$. We then split $\eta'^k$
until $P(\eta'^K)$  is contained in $W$. The face of $P(\eta'^K)$
containing $\phi^{-1}(\lambda)$ might not itself be a train track
occurring in our partition sequence. However, there is some nearly
complete train track $\sigma^K \in \Sigma^K$ carrying $\lambda$. We
split $\sigma^K$ until its polyhedron of projective measures lies
in the interior of the appropriate faces of both $P(\eta^K)$ and
$P(\eta'^K)$. Then the resulting three train tracks define the
required neighbourhood.

More precisely, by property (Adjacency), there is $\eta'^k\in \Tau_k$
with $\phi^{-1}(\lambda)\subset P(\eta^k)\cap P(\eta'^k)$. By
Theorem~\ref{nesting} and property (Full splits), there are some
$K\geq k$ and $\eta'^K\in \Tau_{K}$ with $\phi^{-1}(\lambda)\subset
P(\eta'^K)\subset W$.
Let $\sigma^K \in \Sigma_K$ be the nearly complete train track carrying $\lambda$.
By property
(Adjacency), the set $N$ obtained from
$$V(\eta^K)\cup V(\eta'^K)\cup (\partial
V(\eta^K)\cap
\partial V(\eta'^K))$$
by removing the $1$--skeleta of $P(\eta^K)$ and $P(\eta'^K)$, is an
open neighbourhood of $\phi^{-1}(\lambda)$.

By Theorem~\ref{nesting} and property (Full splits), there is some
$L\geq 0$ and $\sigma^L\in \Sigma_{L}$ with $\phi^{-1}(\lambda)\subset
P(\sigma^L)\subset N$. By Lemma~\ref{key lemma}(i), we have
$\phi^{-1}(\lambda)\subset V(\sigma^L)$. We put $V=V(\eta^K)\cup
V(\sigma^L)\cup V(\eta'^K)$. Then we have $\phi^{-1}(\lambda)\subset
V\subset W$.

Since $P(\sigma^L)$ is $2$--dimensional in $\partial V(\eta^K)$ and
$\partial V(\eta'^K)$, we have that $V(\sigma^L)$ is open in
$\partial V(\eta^K)\cap \partial V(\eta'^K)$. Hence each point of
$V(\sigma^L)$ lies in the interior of $V$, and we conclude that $V$ is
open, as desired. \qed

\section{Almost filling paths}
\label{section Almost filling paths} In this section we give some
account on Gabai's method of constructing paths in $\EL$. This
discussion is valid for any surface $S_{g,p}$ with $\xi=3g-3+p\geq
2$. Gabai's main result is the following.

\begin{theorem}[{\cite[Theorem 0.1]{G}}]
\label{theorem gabai}
$\EL$ is connected, path connected and cyclic.
\end{theorem}

Here we give some details on Gabai's construction, which we need in the proof of Theorem
\ref{main}. Recall \cite{G} that a geodesic lamination $\lambda$
is \emph{almost minimal almost filling} if it has the form $\lambda=\lambda^*\cup\gamma$ where $\lambda^*$ has no isolated leaves, the closed (with respect to the path metric) complement of $\lambda^*$ supports at most one simple
closed geodesic, and $\gamma$ is either this geodesic or is empty.
We denote by $\AML\supset \EL$ the set of all almost minimal almost
filling geodesic laminations.

Gabai uses \emph{PL almost filling} paths $h\colon
I=[0,1]\rightarrow \PML$ satisfying $\phi(h(t))\in \AML,\
\phi(h(0)),\phi(h(1))\in \EL$ and some additional properties
satisfied by generic PL paths. We do not recall these properties,
since we use only the combination of the following results. We
assume that $\PML$ is equipped with a fixed metric, and we say that
two points in $\PML$ are \emph{$\e$--close} if they are distance at
most $\e$ in this metric.

\begin{lemma}[{\cite[Lemma 2.9]{G}}]
\label{existence of pl filling paths} Let $h\colon I\rightarrow
\PML$ be a path with $\phi(h(0)),\phi(h(1))\in \EL$. Then for
any $\e>0$ there is a PL almost filling path $h'\colon I\rightarrow
\PML$ with the same endpoints and such that $h'(t)$ is $\e$--close to $h(t)$, for
all $t\in I$.
\end{lemma}

We now fix a hyperbolic metric on $S_{g,p}$ and consider geodesic
laminations as subsets of the projective tangent bundle of
$S_{g,p}$. 
The hyperbolic metric on $S_{g,p}$ induces a natural (Sasaki) metric on the
projective tangent bundle.
For a geodesic lamination $\lambda$, we denote by
$N^{PT}_\e(\lambda)$ its $\e$--neighbourhood in this metric. The key
element of the proof of Theorem~\ref{theorem gabai} is the following
crucial result.

\begin{lemma}[{\cite[Lemma 5.1]{G}}]
\label{Gabai} If $h\colon I\rightarrow \PML$ is a PL almost filling
path, $\e>0,\delta>0$, then there exists a path $g\colon I
\rightarrow \EL$ with $g(0)=\phi(h(0)),\ g(1)=\phi(h(1))$ such that
for each $t\in [0,1]$ there exists $s\in I$ with $|s-t|<\delta$
satisfying $$h(s)^*\subset N^{PT}_\e(\tilde{g}(t)),$$ for some
diagonal extension $\tilde{g}(t)$ of $g(t)$.
\end{lemma}

We also need the following lemma, which roughly says that for $h$ and $g$
as in the assertion of Lemma~\ref{Gabai}, the preimage
$\phi^{-1}(g(I))$ is not far away from $h(I)$ in $\PML$. We restrict
to the case of the five-punctured sphere (although a version of this
result is true in general). This way we may choose a good partition
sequence $\S$ (see Definition~\ref{partition sequence}). We denote
by $\mathcal{V}(\S),\,\mathcal{U}(\S)$ its associated families of
open sets in $\PML$ and $\EL$, respectively (see Theorem~\ref{tracks
form basis} and Corollary~\ref{properties of Us}).

\begin{lemma}
\label{local perturbing} Let $\lambda_0\in U\in \mathcal{U}(\S)$.
Then there is $U'\in \mathcal{U}(\S)$ with $\lambda_0\in U'$ and
$\e>0$ satisfying the following.
\item(i)
Let $\mu\in \Psi(U')$ with $\phi(\mu)\in \AML$. If
$\phi(\mu)^*$ lies in $N^{PT}_\e(\tilde{\lambda})$ for some diagonal extension
$\tilde{\lambda}$ of $\lambda\in \EL$, then $\lambda\in U$.
\item(ii)
Let $\lambda\in U'$, and $\mu\in \PML$ with $\phi(\mu)\in \AML$.
If $\phi(\mu)^*$ lies in $N^{PT}_\e(\tilde{\lambda})$ for some diagonal
extension $\tilde{\lambda}$ of $\lambda$, then $\mu\in \Psi(U)$.
\end{lemma}
\proof Part (i) is proved in course of the proof of \cite[Theorem
6.1]{G} (local path connectedness of $\EL$).

For part (ii), let $V=\Psi(U)$. By Theorem~\ref{tracks form
basis}(Fineness), there are neighbourhoods $V_1,V_2,V'\in
\mathcal{V}(\S)$ of $\phi^{-1}(\lambda_0)$ satisfying
$\overline{V_1}\subset V, \ \overline{V}_2\subset V_1$, and
$\overline{V}'\subset V_2$ (see Figure~\ref{fig:nesting}).
\begin{figure}
 \centering
  \scalebox{0.75}{
   \psfrag{v1}{$V_1$}
   \psfrag{v2}{$V_2$}
   \psfrag{v}{$V$}
   \psfrag{vv}{$V'$}
   \psfrag{phi}{$\phi^{-1}(\lambda)$}
   \psfrag{mu}{$\mu$}
   \psfrag{nu}{$\nu$}
  \includegraphics[width=0.6\textwidth]{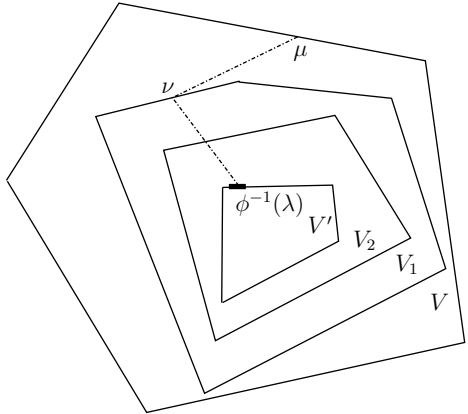} }
  \caption{Nested neighbourhoods with possible position of $\mu$ and $\phi^{-1}(\lambda)$.}
  \label{fig:nesting}
\end{figure}
We prove that $U'=\phi(V'\cap \MPML)$ satisfies the assertion of the
lemma.

First we claim that if we have $\mu\in \PML \setminus V$ with
$\phi(\mu)\in \AML$, then there is a projective measured lamination
$\nu\in\PML\setminus V_1$ with support $\phi(\mu)^*$. Indeed, if
$\phi(\mu)$ is an ending lamination, then we can take $\nu=\mu$.
Otherwise we have $\phi(\mu)=\phi(\mu)^*\cup \gamma$ for some simple
closed geodesic $\gamma$. Let $\nu$ be the projective measured
lamination with support $\phi(\mu)^*$ obtained by restricting the
measure $\mu$ to $\phi(\mu)^*$. By Remark~\ref{no crossing
boundaries}, the interval of projective measured laminations
determined by $\nu$ and $\gamma$ is contained in $\PML\setminus
V_1$. This justifies the claim.

By Remark~\ref{no crossing boundaries}, the supports of any pair of
projective measured laminations in $\overline{V}'$ and
$\PML\setminus V_1$ intersect transversally. Observe that
$\overline{V}'$ and $\PML\setminus V_1$ are compact in $\PML$. By
super convergence of supports (\cite[Proposition 3.2(i)]{G}), there
is $\delta>0$ satisfying the following. For any $\lambda\in
\phi(\overline{V}'\cap \MPML)$ and $\nu\in\PML\setminus V_1$, the
maximal angle of intersection between $\lambda$ and $\phi(\nu)$ is
at least $\delta$. If we pick $\e$ sufficiently small, and we have
$\phi(\nu)=\phi(\mu)^*$, this violates $\phi(\mu)^*\in
N^{PT}_\e(\tilde{\lambda})$, for any diagonal extension
$\tilde{\lambda}$ of $\lambda$. \qed

\section{Proof of the main theorem}
\label{section Outline of the proof}

Our goal is to prove Theorem~\ref{main}, i.e.\ to verify that
$\mathcal{EL}(\Sp)$ is homeomorphic to the N\"obeling curve. By
Theorem~\ref{characterization_easy}, in order to do this we must
show that $\EL$ is a Polish space, that it is connected, locally
path connected, of dimension $1$, and satisfies the locally finite
$1$--discs property.

To see that $\EL$ is separable note that the orbit of any ending
lamination under the action of the mapping class group is dense in
$\EL$ (see e.g.\ \cite[Corollary 4.2]{Ham2}). Because the mapping
class group is finitely generated, this orbit is countable. Since
$\EL$ is homeomorphic to the Gromov boundary of the curve complex
(\cite[Theorem 1.3]{Kla}, compare \cite[Section 1]{Ham}), it carries
a metric defined using the Gromov product (see Bridson--Haefliger
\cite[Chapter III.H]{BH} and Bonk--Schramm \cite[Section 6]{BS}).
This metric is complete by \cite[Proposition 6.2]{BS}. Hence $\EL$
is a Polish space.

\medskip

By Theorem~\ref{theorem gabai}, $\EL$ is connected and locally path
connected.

\medskip

Now we prove that $\EL$ is of dimension $1$. Since there are paths
in $\EL$ (Theorem~\ref{theorem gabai}), it is not of dimension $0$.
In order to prove that $\EL$ is of dimension at most $1$, we need to
check that any point in $\EL$ has a neighbourhood basis with
boundaries of dimension $0$. Let $\S=((\Tau_k,\Sigma_k))_k$ be a
good partition sequence, guaranteed by Theorem~\ref{tracks form
basis}. Thus, by Corollary~\ref{properties of Us}, it is enough to
prove that the boundary of any $U\in\mathcal{U}(\S)$ is of dimension $0$.

The boundary of any $U\in\mathcal{U}(\S)$ is contained in a union of
boundaries of up to three $U(\tau)$'s, where $\tau\in \Tau_k\cup
\Sigma_k$, for some $k$'s. For $\sigma\in \Sigma_k$ the set
$U(\sigma)$ is closed (Lemma~\ref{key lemma}(iii)), so that
$\mathrm{Fr}\, U(\sigma)\subset U(\sigma)$. For fixed $k$, all
$U(\eta)$ with $\eta\in \Tau_k$ are open and disjoint. Hence, if we
denote $X_k=\bigcup_{\sigma\in\Sigma_k}U(\sigma)$, we have
$\mathrm{Fr }\,U(\eta)\subset X_k$. By property (Subdivision), for
all $K\geq k$ we have $X_K\supset X_k$. Hence the boundary of any
$U\in\mathcal{U}(\S)$ is contained in $X_K$, for sufficiently large
$K$. Thus it is enough to prove that for each $k$ the set
$X_k=\bigcup_{\sigma\in\Sigma_k} U(\sigma)\subset \EL$ is of
dimension $0$.

We obtain a neighbourhood basis for each $\lambda\in X_k$ by
restricting to $X_k$ the sets $U'\ni \lambda$ from the open
neighbourhood basis $\mathcal{U}(\S)$. By Remark~\ref{discarding},
we may assume that $U'=U(\eta)$ or $U'=U(\eta_1)\cup U(\sigma) \cup
U(\eta_2)$ where $\eta$ or all of $\eta_1, \sigma, \eta_2$ lie in
$\Tau_K$'s and $\Sigma_K$'s with $K$'s at least $k$. Since for
$K\geq k$ we have $\EL\setminus X_K\subset \EL \setminus X_k$, we
get that all $U(\eta)$, with $\eta\in \Tau_K$, are disjoint from
$X_k$. Hence $U'$ is of the form $U(\eta_1)\cup U(\sigma) \cup
U(\eta_2)$ and $U'\cap X_k\subset U(\sigma)$, for some $\sigma\in
\Sigma_K$.

We now prove that actually we have the equality $U'\cap X_k=
U(\sigma)$. By property (Subdivision), $U(\sigma)$ is contained in
some $U(\tau)$, where $\tau\in \Tau_k\cup\Sigma_k$. Since
$\lambda\in U(\sigma)\subset U(\tau)$ and $\lambda\in X_k$, we have
$\tau\in \Sigma_k$. Hence $U(\sigma)\subset X_k$. Summarizing, the
restriction of $U'$ to $X_k$ equals $U(\sigma)$. By Lemma~\ref{key
lemma}(iii), $U(\sigma)$ is closed in $\EL$, hence it is also closed
in $X_k$. Moreover, $U(\sigma)$ is also open in $X_k$, since its
complement is a finite union of some disjoint closed $U(\sigma')$
with $\sigma'\in \Sigma_K$. Thus the boundary of $U(\sigma)$ in
$X_k$ is empty, as desired.

\medskip

In order to finish the proof of Theorem~\ref{main}, it remains to
prove the following, which we do in Section~\ref{section
Universality}.

\begin{prop}
\label{main proposition}
 $\EL(S_{0,5})$ satisfies the locally finite
$1$--discs property.
\end{prop}

\section{Universality}
\label{section Universality} In this section we prove Proposition
\ref{main proposition}, which completes the proof of Theorem
\ref{main}.

\medskip

We have to prove that for any family of paths $f_n \colon I
\rightarrow \EL$, where $n\in \N$, and any open cover $\mathcal{U}$
of $\EL$, there are paths $g_n\colon I \rightarrow \EL$ such that
\begin{description}
\item[(Local finiteness)]
 for each $\lambda\in \EL$ there is a neighbourhood $U\ni \lambda$
 satisfying
$g_n(I)\cap U=\emptyset$ for sufficiently large $n$,
\item[(Approximation)]
for each $t\in I, n\in \N,$ there is $U\in \mathcal{U}$ such that
both $f_n(t)$ and $g_n(t)$ lie in $U$.
\end{description}

Because the proof is technically involved, as an illustration we
prove the following.

\begin{prop} The N\"obeling curve $N^3_1\subset \R^3$ satisfies the locally finite $1$--discs
property.
\end{prop}
\proof We learned the idea of this proof from Andrzej Nag\'orko.

We say that a cube $I_1\times I_2\times I_3\subset \R^3$ is
\emph{$m$--diadic} if the lengths of all $I_i$ equal $\frac{1}{2^m}$
and the endpoints of $I_i$ lie in $\frac{1}{2^m}\Z$.

Assume first, for simplicity, that we are in the special case where
there is $m>0$ such that the open cover $\mathcal{U}$ consists of
the interiors of unions of pairs of adjacent $m$--diadic cubes. Let
$\Gamma\subset \R^3$ be the closed set which is the union of all
lines parallel to coordinate axes with the fixed coordinate in
$\frac{1}{2^m}\Z$. In other words $\Gamma$ is the union of
$1$--skeleta of all $m$--diadic cubes. Observe that $\Gamma$ is
disjoint from $N^3_1$.

Let $f_n \colon I \rightarrow N^3_1$ be the family of paths which we
want to approximate. We describe the construction of $g_n$ for a
fixed $n\in \N$.

There is a partition $J_1\cup I_1\cup\ldots \cup I_{l-1}\cup J_l$ of
$I$ into closed intervals with disjoint interiors, satisfying the
following. There are $m$--diadic open cubes $V_k$, where $0\leq
k\leq l$, satisfying $f_n(I_k)\subset V_k$ and $f_n(J_k)\subset
\mathrm{int}\, \overline{V_{k-1}\cup V_k}$. Denote the endpoints of
$J_k$ by $s_k,t_k$. We can assume $l\geq 1$.

For each pair of adjacent cubes $V_{k-1}, V_k$ we denote by
$\Gamma_k\subset \Gamma$ the square loop which is the intersection
of $\Gamma$ with $\partial V_{k-1}\cap \partial V_k$. For $A\subset
\R^3$ and $\delta>0$ we denote by $N_\delta(A)$ the open
$\delta$--neighbourhood of $A$ in $\R^3$.

For each $1< k\leq l$ we choose some $p_k\in
V_{k-1}\cap N_\frac{1}{n}(\Gamma_k)$ satisfying $p_k\in N^3_1$. We
put $g_n(s_k)=p_k$. Analogously, for each $1\leq k< l$ we choose
some $q_k\in V_k\cap N_\frac{1}{n}(\Gamma_k)$
satisfying $q_k\in N^3_1$. We put $g_n(t_k)=q_k$. We also put
$g_n(0)=q_1$ and $g_n(1)=p_l$.

For $0<k<l$ we choose paths $h^{I_k}$ between $q_k$ and $p_{k+1}$ in
the open sets $V_k\cap N_\frac{1}{n}(\Gamma)$. This is possible,
since the latter sets are neighborhoods of $1$--skeleta of $V_k$,
hence they are path connected. We define the path $g_n$ on $I_k$ by
slightly perturbing $h^{I_k}$ relative the endpoints so that we
obtain paths in $N^3_1$.

Similarly, for $1\leq k\leq l$ we choose paths $h^{J_k}$ between
$p_k$ and $q_k$ in the open sets $\mathrm{int}\,
\overline{V_{k-1}\cup V_k}\cap N_\frac{1}{n}(\Gamma_k)$. The latter
sets are path connected because $\Gamma_k$ are $1$--spheres. We
define the path $g_n$ on $J_k$ by slightly perturbing $h^{J_k}$
relative the endpoints so that we obtain paths in $N^3_1$.

By construction, paths $g_n$ are $\mathcal{U}$--close to $f_n$,
which means that they satisfy property (Approximation). Moreover,
for each $n$ the image of the path $g_n$ is contained in
$N_\frac{1}{n}(\Gamma)$, where $\Gamma$ is a closed set disjoint
from $N^3_1$. This yields property (Local finiteness). This ends the
proof in the case of special $\mathcal{U}$.

\medskip

In general, we may only assume that $\mathcal{U}$ consists of the
interiors of unions of pairs of adjacent $m$--diadic cubes without
the assumption that $m$ is fixed. In other words the cubes might be
arbitrarily small. However, we can at least assume that no element
of $\mathcal{U}$ is properly contained in another one. We also note
the property that if two open diadic cubes intersect, then one of
them is contained in the other. We define (the ``attracting
grid'') $\Gamma$ as the complement in $\R^3$ of the union of all
elements of $\mathcal{U}$.

\medskip\par\noindent\textbf{Claim.}\ignorespaces
\
\begin{enumerate}[(i)]
\item
For each pair of adjacent cubes $V_1,V_2$ with $\mathrm{int} \,
\overline{V_1\cup V_2}\in \mathcal{U}$, the square loop which is the
boundary of the common face of $V_1$ and $V_2$ is contained in
$\Gamma$.

\item
Let $V$ be a maximal (open) cube among all cube pairs from
$\mathcal{U}$. Then $\partial V\cap \Gamma$ is connected (and
non-empty).
\end{enumerate}

Assertion (i) of the claim follows directly from the maximality
assumption on the elements of $\mathcal{U}$. For assertion (ii)
observe first that $\partial V$ is a $2$--sphere. We obtain
$\partial V\cap \Gamma$ by removing from $\partial V$ the
intersections with elements of $\mathcal{U}$. By maximality
assumption on $V$, these elements have the form $\mathrm{int} \,
\overline{V_1\cup V_2}$, where $V_1\subset V$ and $V_2\subset
\R^3\setminus V$. Each intersection of such a set with $\partial V$
is an open $2$--disc. By the maximality assumption on the elements
of $\mathcal{U}$, all those $2$--discs are disjoint. Hence $\partial
V\cap \Gamma$ is obtained from a $2$--sphere by removing a disjoint
union of open $2$--discs, which yields assertion (ii).

\medskip

We leave it to the reader to verify that the claim allows to perform
the same argument as in the special case. \qed

\medskip
We are now prepared for the following.

\medskip\par\noindent\textbf{Proof of Proposition~\ref{main proposition}.}\ignorespaces
\ By Theorem~\ref{tracks form basis}, we may assume that
$\mathcal{U}\subset \mathcal{U}(\S)$, where $\mathcal{U}(\S)$ is the
open neighbourhood basis coming from some fixed good partition
sequence $\S=((\Tau_k, \Sigma_k))_k$ (see Definition~\ref{partition
sequence} and Corollary~\ref{properties of Us}). Let
$\mathcal{U}'\subset \mathcal{U}(\S)$ be an open cover which is a
refinement of $\mathcal{U}$ satisfying the assertion of Lemma
\ref{local perturbing}(i). In other words, we require that for any
$U'\in \mathcal{U}'$, there is $U=U(U')\in \mathcal{U}$ and
$\e=\e(U')>0$, so that we have the following. For any $\mu\in
V'=\Psi(U')$ with $\phi(\mu)\in \AML$, if $\lambda\in \EL$ and
$\phi(\mu)^*\in N^{PT}_\e(\tilde{\lambda})$ for some diagonal
extension $\tilde{\lambda}$ of $\lambda$, then $\lambda\in U$.
Without loss of generality we may assume that whenever $U(\eta_1)
\cup U(\sigma) \cup V(\eta_2)$ belongs to $\mathcal{U}'$, then also
$U(\eta_1)$ and $U(\eta_2)$ belong to $\mathcal{U}'$.

We say that a train track $\tau\in \Tau_k\cup \Sigma_k$
\emph{participates in} $\mathcal{U}'$, if $U(\tau)\in \mathcal{U}'$
or $\tau$ equals $\sigma$ for $U(\eta_1)\cup U(\sigma)\cup
U(\eta_2)\in \mathcal{U}'$. Let $\Tau'\subset \bigcup_k\Tau_k$ be
the family of all complete train tracks $\eta$ participating in
$\mathcal{U}'$ with maximal $V(\eta)$ with respect to inclusion. In
other words, we take all complete train tracks participating in
$\mathcal{U}'$ and remove those $\eta$ whose $V(\eta)$ is properly
contained in some $V(\eta')$, where $\eta'$ is also participating in
$\mathcal{U}'$. Note that by property (Subdivision) $V(\eta)$ and
$V(\eta')$ can only intersect if one is contained in the other. We
denote the family of $U(\eta)$ over $\eta\in \Tau'$ by
$\mathcal{U}(\Tau')$. Denote
$\mathcal{V}(\Tau')=\Psi(\mathcal{U}(\Tau'))$. Since $V(\eta)$ were
required to be maximal, the elements of $\mathcal{V}(\Tau')$ are
pairwise disjoint. Hence the elements of $\mathcal{U}(\Tau')$ are
also pairwise disjoint.

Let $\Sigma'\subset \bigcup_k\Sigma_k$ be the family of all nearly
complete train tracks $\sigma$ participating in $\mathcal{U}'$ with
maximal $V(\sigma)$ with respect to inclusion, among all $V(\tau)$ with
$\tau$ participating in $\mathcal{U}'$. We denote the family of
$U(\sigma)$ over $\sigma\in \Sigma'$ by $\mathcal{U}(\Sigma')$ and
we put $\mathcal{V}(\Sigma')=\Psi(\mathcal{U}(\Sigma'))$. The
elements of $\mathcal{V}(\Sigma')$ are pairwise disjoint and
disjoint from the elements of $\mathcal{V}(\Tau')$. Hence the
elements of $\mathcal{U}(\Sigma')$ are also pairwise disjoint and
disjoint from the elements of $\mathcal{U}(\Tau')$.

Let $\Gamma\subset \PML$ be the closed set which is the complement
of the union of all sets in $\mathcal{V}'=\Psi(\mathcal{U}')$. We have
$\Gamma \cap \PEML=\emptyset$.

\medskip\par\noindent\textbf{Claim 1.}\ignorespaces
\
\begin{enumerate}[(i)]
\item
For any $\sigma\in \Sigma'$, we have $\partial V(\sigma)\subset
\Gamma$.
\item
For any $\eta\in\Tau'$ the set $\partial V(\eta) \cap \Gamma$ is
connected and non-empty.
\end{enumerate}

\medskip\par\noindent\textbf{Proof of Claim 1.}\ignorespaces
\begin{enumerate}[(i)]
\item Let $\mu\in \partial V(\sigma)$. If
$\mu\notin \Gamma$, then there is $V'\in \mathcal{V}'$ with $\mu\in
V'$. Since $V'$ is open in $\PML$, it intersects $V(\sigma)$. The
set $V'$ is of the form $V'=V(\eta)$ or $V'=V(\eta_1)\cup
V(\sigma')\cup V(\eta_2)$. Thus $V(\sigma)$ intersects $V(\tau)$,
for $\tau$ equal to one of $\eta, \eta_i,\sigma'$. Since $\sigma\in
\Sigma'$, we have $V(\tau)\subset V(\sigma)$. Hence $\tau$ is a
nearly complete train track, and therefore $V'=V(\eta_1)\cup
V(\sigma')\cup V(\eta_2)$ where $\sigma'$ is equal to $\tau$. Since
$\sigma\in\Sigma'$, we have $V(\sigma)\supset V(\sigma')$. By
hypothesis $\mu$ is outside of $V(\sigma)$, hence it lies in
$V(\eta_i)$ for some $i$. But then $V(\eta_i)$ intersects
$V(\sigma)$ and like before we get $V(\eta_i)\subset V(\sigma)$,
which is a contradiction.

\item First note that $\partial V(\eta)$ is a
$2$--sphere. $\partial V(\eta)$ is disjoint from any $V(\eta')$, for
$\eta'$ participating in $\mathcal{U'}$: otherwise $V(\eta')$
intersects $V(\eta)$ and by maximality of $V(\eta)$ (since $\eta\in
\Tau'$) we have $V(\eta')\subset V(\eta)$, which is a contradiction.
Hence, if for some $V'\in \mathcal{V}'$ the intersection $\partial
V(\eta)\cap V'$ is non-empty, then $V'=V(\eta_1)\cup V(\sigma)\cup
V(\eta_2)$, where $\sigma\in \Sigma'$, and $\partial V(\eta)\cap
V'\subset V(\sigma)$. Moreover, since $V'$ is open, we have that
$V(\sigma)$ or one of $V(\eta_i)$ intersects $V(\eta)$. Since
$\sigma\in \Sigma'$, this must be one of $V(\eta_i)$.  Because
$\eta\in \Tau'$, we then have $V(\eta_i)\subset V(\eta)$. In
particular, $P(\sigma) \subset
\partial V(\eta_i) \subset P(\eta)$. Again, since $\sigma\in
\Sigma'$ we have $P(\sigma)\subset \partial V(\eta)$. Summarizing,
for any $V'\in \mathcal{V}'$ we have $\partial V(\eta)\cap
V'=V(\sigma)$, for some $\sigma\in \Sigma'$, which is an open
$2$--disc. Hence $\partial V(\eta)\cap \Gamma$ is obtained from the
$2$--sphere $\partial V(\eta)$ by removing a (possibly infinite)
union of disjoint open $2$--discs.
\end{enumerate}
\qed
\medskip

Let $f_n\colon I\rightarrow \EL$, where $n\in \N$, be the family of
paths which we want to approximate. We now independently construct
the paths $g_n$. To this end, we fix $n\in\N$ and note the
following.

\medskip\par\noindent\textbf{Claim 2.}\ignorespaces
\ There is a partition $I_0\cup J_1\cup I_1\cup\ldots \cup J_l\cup I_l$ of $I$
into closed intervals with disjoint interiors, with possibly empty
$I_0, I_l$, satisfying the following (see Figure~\ref{fig:path}).
\begin{itemize}
\item
$f_n(I_k)\subset U'_k$, for some $U'_k\in \mathcal{U}(\Tau')\subset
\mathcal{U}'$, if $I_k$ is non-empty, where $0\leq k\leq l$.

\item
$f_n(J_k)\subset \hat{U}_k$, where $\hat{U}_k=U(\eta^1_k)\cup
U(\sigma_k)\cup U(\eta^2_k)\in \mathcal{U'}$ with $\sigma_k\in
\Sigma'$. Moreover, for $j=k-1,k$ there is $i$ such that
$U(\eta^i_k)\subset U'_j$, if $I_j$ is non-empty, where $1\leq k\leq
l$.
\end{itemize}

\medskip\par\noindent\textbf{Proof of Claim 2.}\ignorespaces
\ For the proof it is convenient to introduce the following
terminology. We call an element $U$ of $\mathcal{U}(\S)$ a
\emph{vertex block} if it is of the form $U(\eta)$ for some complete
train track $\eta$. We call the other elements of $\mathcal{U}(\S)$
\emph{edge blocks}.
\begin{figure}[htbp!]
 \centering
  \scalebox{0.75}{
   \psfrag{v1}{$V'_0$}
   \psfrag{v2}{$V'_1$}
   \psfrag{v3}{$V'_2$}
   \psfrag{v4}{$V'_3$}
   \psfrag{v5}{$V'_4$}
   \psfrag{w1}{$\hat{V}_1$}
   \psfrag{w2}{$\hat{V}_2$}
   \psfrag{w3}{$\hat{V}_3$}
   \psfrag{w4}{$\hat{V}_4$}
  \includegraphics[width=\textwidth]{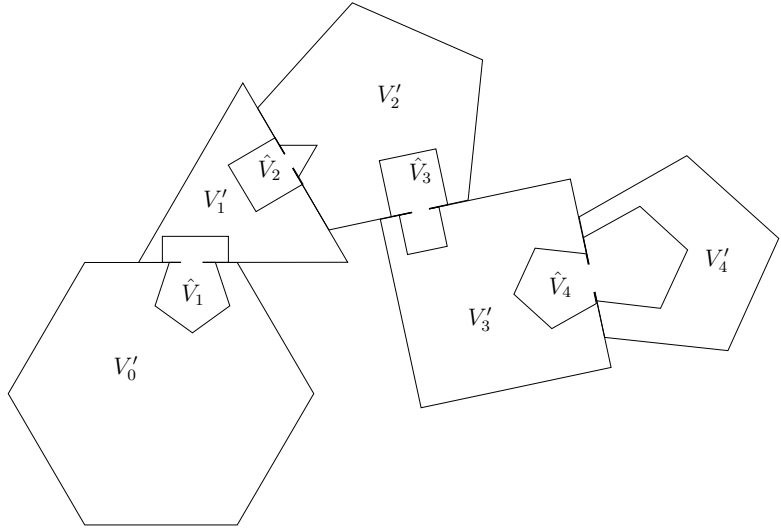} }
  \caption{Combinatorics of vertex blocks and edge blocks.}
  \label{fig:path}
\end{figure}

Consider now the family $\mathcal{Y}$, which is the union of all
vertex blocks in $\mathcal{U}(\Tau')\subset \mathcal{U}'$ and edge
blocks $U(\eta^1)\cup U(\sigma)\cup U(\eta^2)\in \mathcal{U'}$ with
$\sigma\in \Sigma'$, where we pick only one such set for each
$\sigma$. Observe that
$\mathcal{Y}$ forms an open cover of $\EL$. By compactness of $I$, a
finite subset of $\mathcal{Y}$ covers $f_n(I)$. In particular, there
is a partition $\mathcal{I}$ of $I$ into finitely many nontrivial
closed intervals with disjoint interiors so that each interval is
mapped into a set of $\mathcal{Y}$. Observe that two consecutive
intervals in $\mathcal{I}$ cannot be mapped into different vertex
blocks, since the latter are disjoint. Moreover, if two consecutive
intervals $I_-, I_+\in \mathcal{I}$ are mapped into edge blocks
$U(\eta_-^1)\cup U(\sigma_-)\cup U(\eta_-^2),\ U(\eta_+^1)\cup
U(\sigma_+)\cup U(\eta_+^2)\in \mathcal{U'}$, then we have the
following. Since $\sigma_-\neq \sigma_+$, we have that $f_n(I_-\cap
I_+)$ lies in, say, $U(\eta_-^1)\cap U(\eta_+^1)$, hence
$U(\eta_-^1)$ and $U(\eta_+^1)$ are contained in the same vertex
block $U'\in \mathcal{U}(\Tau')$. We can then represent $I_-\cup
I_+=I_-'\cup J\cup I_+'$ with $I_-'\subset I_-, \ I_+'\subset I_+$
and $f_n(J)\subset U'$, where $J$ is nontrivial. To conclude this
discussion, we can assume that the intervals in $\mathcal{I}$ are
mapped alternatively into vertex blocks and edge blocks of
$\mathcal{Y}$. Furthermore, observe that for each pair of
consecutive intervals in $\mathcal{I}$ mapped by $f_n$ to $U(\eta)
\in \mathcal{Y}$ and $U(\eta^1)\cup U(\sigma)\cup U(\eta^2) \in
\mathcal{Y}$ there is some $i$ with $U(\eta^i)\subset U(\eta)$. This
gives rise to $I_k,J_k$ as required.\qed

\medskip
From now on we fix the objects and the notation as in Claim 2.
Before we describe the construction of the path $g_n$, note that to
guarantee property (Approximation) it suffices that $g_n$ satisfies the following
two properties.
\begin{description}
\item[(Approximation i)]
For each $0\leq k\leq l$ such that $I_k$ is non-empty we have
$g_n(I_k)\subset U(U'_k)$.

\item[(Approximation ii)]
For each $1\leq k\leq l$ we have $g_n(J_k)\subset U(\hat{U}_k)$.
\end{description}
\begin{figure}[htbp!]
 \centering
  \scalebox{0.6}{
   \psfrag{g}{$N_{\frac{1}{n}}(\Gamma)$}
   \psfrag{pi}{$p_k$}
   \psfrag{qi}{$q_k$}
   \psfrag{qii}{$p_{k+1}$}
   \psfrag{hj}{$h^{J_k}$}
   \psfrag{hi}{$h^{I_k}$}
  \includegraphics[width=\textwidth]{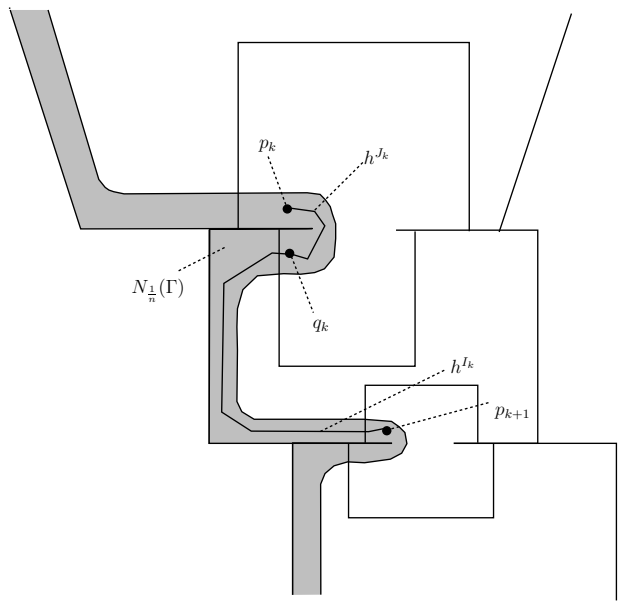} }
  \caption{Construction of approximating paths attracted by $\Gamma$.}
  \label{fig:gamma}
\end{figure}

\medskip
At this point we can finally define the path $g_n$. Denote
$V'_k=\Psi(U'_k)$ and $\hat{V}_k=\Psi(\hat{U}_k)$. If $l=0$, we
choose any point $p\in V'_0\cap N_ \frac{1}{n}(\Gamma)$ satisfying
$\phi(p)\in \EL$. This is possible since the open set $V'_0\cap N_
\frac{1}{n}(\Gamma)$ is non-empty by Claim 1(ii) and $\MPML$ is
dense in $\PML$. We put $g_n(I) \equiv p$, which obviously satisfies
properties (Approximation i) and (Approximation ii).

From now on we assume $l\geq 1$. Denote the endpoints of $J_k$ by
$s_k,t_k$. First we claim that for any $1\leq k \leq l$ with $s_k
\neq 0$ the open set $V'_{k-1}\cap \hat{V}_k\cap
N_\frac{1}{n}(\partial V(\sigma_k))$ is non-empty. Indeed, there is
an $i$ satisfying $V(\eta_k^i) \subset V'_{k-1}$. Then we have
$\partial V(\sigma_k) \subset \partial V(\eta_k^i) \subset
\overline{V}_{k-1}'$. Hence $\partial V(\sigma_k)$ is contained in
the closure of $V'_{k-1}\cap \hat{V}_k = V(\eta^i_k)$ and the claim
follows.

Thus for each $1\leq k\leq l$ with $s_k\neq 0$ we can choose some
$p_k\in V'_{k-1}\cap \hat{V}_k\cap N_\frac{1}{n}(\partial
V(\sigma_k))$ satisfying $\phi(p_k)\in \EL$. We put
$g_n(s_k)=\phi(p_k)$. Analogously, for each $1\leq k\leq l$ with
$t_k\neq 1$ we choose some $q_k\in V'_k\cap \hat{V}_k\cap
N_\frac{1}{n}(\partial V(\sigma_k))$ satisfying $\phi(q_k)\in \EL$.
We put $g_n(t_k)=\phi(q_k)$. By Claim 1(i) all $p_k,q_k$ lie in
$N_\frac{1}{n}(\Gamma)$.

If $s_1=0$ (i.e.\ if $I_0$ is empty), then we put $g_n(0)=q_1$,
otherwise we put $g_n(I_0)\equiv p_1$. Analogously, if $t_l=1$
(i.e.\ if $I_l$ is empty), then we put $g_n(1)=p_l$, otherwise we
put $g_n(I_l)\equiv q_l$.

\medskip

For $0<k<l$, we choose some PL almost filling paths $h^{I_k}$ between $q_k$ and $p_{k+1}$ in
the sets $V'_k\cap N_\frac{1}{n}(\Gamma)$ (see Figure~\ref{fig:gamma}).
This is possible, since the latter sets are open and path connected by
Claim 1(ii). To each $h^{I_k}$ we apply Lemma~\ref{Gabai} with $\e =
\e'_n = \min \{\e(U'_k),\ \frac{1}{n}\}$ (and any $\delta$). We obtain
paths $g_n^{I_k}\colon I_k\rightarrow\EL$ with endpoints $\phi(q_k),
\phi(p_{k+1})$, such that the image of $g_n^{I_k}$ lies in
$U(U'_k)$. We define $g_n$ on $I_k$ to be equal to $g_n^{I_k}$,
which gives property (Approximation i).

Similarly, for $1\leq k\leq l$ we choose some PL almost
filling paths $h^{J_k}$ between $p_k$ and $q_k$ in the sets $\hat{V}_k\cap N_\frac{1}{n}(\partial V(\sigma_k))$ (also see Figure~\ref{fig:gamma}).
This is possible since the latter sets are open and path
connected (because $\partial V(\sigma_k)$ are $1$--spheres).
 To each
$h^{J_k}$ we apply Lemma~\ref{Gabai} with $\e = \hat{\e}_n = \min
\{\e(\hat{U}_k),\ \frac{1}{n}\}$. We obtain paths
$g_n^{J_k}\colon J_k\rightarrow\EL$ with endpoints $\phi(p_k),
\phi(q_k)$, such that the image of $g_n^{J_k}$ lies in
$U(\hat{U}_k)$. We define $g_n$ on $J_k$ to be equal $g_n^{J_k}$,
which gives property (Approximation ii).

This concludes the construction of the paths $g_n\colon I
\rightarrow \EL$. By the discussion above they satisfy property
(Approximation).

\medskip

It remains to verify property (Local finiteness). Let $\lambda_0\in
\EL$. Let $V\in \mathcal{V}(\S)$ be a neighbourhood of
$\phi^{-1}(\lambda_0)$ such that its closure is disjoint from
$\Gamma$ (guaranteed by Theorem~\ref{tracks form basis}(Fineness)).
Put $U=\phi(V\cap \MPML)\in \mathcal{U}(\S)$. Let $U'\subset U$ and
$\e>0$ be as in the assertion of Lemma~\ref{local perturbing}(ii)
applied to $U$ and $\lambda_0$. For sufficiently large $n$ we have
that $V$ is outside $N_\frac{1}{n}(\Gamma)$ and $\frac{1}{n}\leq\e$.
Then both $\hat{\e}_n$ and $\e'_n$ are smaller than $\e$, and
therefore the image of $g_n$ is outside $U'$. \qed


\begin{bibdiv}
\begin{biblist}

\bib{BS}{article}{
   author={Bonk, M.},
   author={Schramm, O.},
   title={Embeddings of Gromov hyperbolic spaces},
   journal={Geom. Funct. Anal.},
   volume={10},
   date={2000},
   number={2},
   pages={266--306}
   }

\bib{Bo}{article}{
   author={Bowers, P. L.},
   title={Dense embeddings of sigma-compact, nowhere locally compact metric
   spaces},
   journal={Proc. Amer. Math. Soc.},
   volume={95},
   date={1985},
   number={1},
   pages={123--130}
   }

\bib{BH}{book}{
   author={Bridson, M. R.},
   author={Haefliger, A.},
   title={Metric spaces of non-positive curvature},
   series={Grundlehren der Mathematischen Wissenschaften [Fundamental
   Principles of Mathematical Sciences]},
   volume={319},
   publisher={Springer-Verlag},
   place={Berlin},
   date={1999},
   pages={xxii+643}
   }


\bib{Cu}{article}{
   author={Curtis, D. W.},
   title={Boundary sets in the Hilbert cube},
   journal={Topology Appl.},
   volume={20},
   date={1985},
   number={3},
   pages={201--221}
   }

\bib{Du}{article}{
   author={Dugundji, J.},
   title={Absolute neighborhood retracts and local connectedness in
   arbitrary metric spaces},
   journal={Compositio Math.},
   volume={13},
   date={1958},
   pages={229--246 (1958)}
   }

\bib{Eng}{book}{
   author={Engelking, R.},
   title={Dimension theory},
   note={Translated from the Polish and revised by the author;
   North-Holland Mathematical Library, 19},
   publisher={North-Holland Publishing Co.},
   place={Amsterdam},
   date={1978},
   pages={x+314 pp. (loose errata)}
   }

\bib{G}{article}{
   author={Gabai, D.},
   title={Almost filling laminations and the connectivity of ending
   lamination space},
   journal={Geom. Topol.},
   volume={13},
   date={2009},
   number={2},
   pages={1017--1041}
   }

\bib{Ham}{article}{
   author={Hamenst{\"a}dt, U.},
   title={Train tracks and the Gromov boundary of the complex of curves},
   conference={
      title={Spaces of Kleinian groups},
   },
   book={
      series={London Math. Soc. Lecture Note Ser.},
      volume={329},
      publisher={Cambridge Univ. Press},
      place={Cambridge},
   },
   date={2006},
   pages={187--207}
   }

\bib{Ham2}{article}{
   author={Hamenst{\"a}dt, U.},
   title={Geometric properties of the mapping class group},
   conference={
      title={Problems on mapping class groups and related topics},
   },
   book={
      series={Proc. Sympos. Pure Math.},
      volume={74},
      publisher={Amer. Math. Soc.},
      place={Providence, RI},
   },
   date={2006},
   pages={215--232}
   }

\bib{Ham_MCG1}{article}{
   author={Hamenst{\"a}dt, U.},
   title={Geometry of the mapping class groups. I. Boundary amenability},
   journal={Invent. Math.},
   volume={175},
   date={2009},
   number={3},
   pages={545--609}}

\bib{KLT}{article}{
   author={Kawamura, K.},
   author={Levin, M.},
   author={Tymchatyn, E. D.},
   title={A characterization of 1-dimensional N\"obeling spaces},
   booktitle={Proceedings of the 12th Summer Conference on General Topology
   and its Applications (North Bay, ON, 1997)},
   journal={Topology Proc.},
   volume={22},
   date={1997},
   number={Summer},
   pages={155--174}
   }

\bib{Kla}{article}{
   author={Klarreich, E.},
   title ={The boundary at inifinity of the curve complex and the
   relative Teichm\"uller space},
   eprint={http://www.msri.org/people/members/klarreic/curvecomplex.ps}
   date={1999}
   }

\bib{LS}{article}{
   author={Leininger, C.},
   author={Schleimer, S.},
   title ={Connectivity of the space of ending laminations},
   status={preprint},
   eprint={arXiv:0801.3058},
   date={2008}
   }

\bib{LMS}{article}{
   author={Leininger, C.},
   author={Mj, M.},
   author={Schleimer, S.},
   title ={Universal Cannon--Thurston maps and the boundary of the curve complex},
   status={preprint},
   eprint={arXiv:0808.3521},
   date  ={2008}
   }

\bib{Luo}{article}{
   author={Luo, F.},
   title={Automorphisms of the complex of curves},
   journal={Topology},
   volume={39},
   date={2000},
   number={2},
   pages={283--298}
   }

\bib{MM}{article}{
   author={Masur, H. A.},
   author={Minsky, Y. N.},
   title={Geometry of the complex of curves. I. Hyperbolicity},
   journal={Invent. Math.},
   volume={138},
   date={1999},
   number={1}
   }

\bib{Mos}{article}{
   author={Mosher, L.},
   title ={Train track expansions of measured foliations},
   date={2003},
   status={unpublished manuscript},
   eprint={http://andromeda.rutgers.edu/~mosher/arationality_03_12_28.pdf}
   }


\bib{N}{article}{
   author={Nag\'orko, A.},
   title ={Characterization and topological rigidity of N\"obeling manifolds},
   date={2006},
   status={submitted}
   note={PhD thesis},
   eprint={arXiv:0602574}
   }

\bib{PH}{book}{
   author={Penner, R. C.},
   author={Harer, J. L.},
   title={Combinatorics of train tracks},
   series={Annals of Mathematics Studies},
   volume={125},
   publisher={Princeton University Press},
   place={Princeton, NJ},
   date={1992},
   pages={xii+216}
   }

\bib{Th}{article}{
   author={Thurston, W. P.},
   title ={The geometry and topology of three-manifolds},
   publisher={Princeton Univ. Math. Depr. Lecture Notes},
   date={1980},
   eprint={http:msri.org/publications/books/gt3m/}
   }

\end{biblist}
\end{bibdiv}
\end{document}